\documentclass[reqno]{amsart}
\usepackage[backend=bibtex, sorting=none, sorting=nyt,maxbibnames=99]{biblatex}  
\renewbibmacro{in:}{}
\bibliography{references.bib}
\usepackage{amssymb}
\usepackage{calrsfs}
\usepackage{graphics,graphicx}
\usepackage{enumerate}
\usepackage{enumitem}
\usepackage{url}
\usepackage{xcolor}
\usepackage[ruled, lined, linesnumbered, longend]{algorithm2e}
\usepackage{hyperref}
\usepackage[justification=centering]{caption}

\newtheorem{theorem}{Theorem}[section]
\newtheorem{lemma}[theorem]{Lemma}

\newtheorem{proposition}[theorem]{Proposition}

\newtheorem{corollary}[theorem]{Corollary}
\numberwithin{equation}{section}

\theoremstyle{remark}
\newtheorem*{remark}{Remark}

\DeclareMathOperator{\li}{li}

\def\reals{\hbox{\rm I\kern-.18em R}}
\def\complexes{\hbox{\rm C\kern-.43em
\vrule depth 0ex height 1.4ex width .05em\kern.41em}}
\def\field{\hbox{\rm I\kern-.18em F}} 

\newenvironment{section*}[2][A]{
  \section*{#2}
  \renewcommand\thesection{#1}
  \setcounter{theorem}{0}}{}

\allowdisplaybreaks

\usepackage{tikz}

\begin{document}

\title[Some explicit results on the sum of a prime and an almost prime]{Some explicit results on the sum of a prime and an almost prime}

\author{Daniel R. Johnston and Valeriia V. Starichkova}
\address{School of Science, The University of New South Wales, Canberra, Australia}
\email{daniel.johnston@adfa.edu.au}
\address{School of Science, The University of New South Wales, Canberra, Australia}
\email{v.starichkova@adfa.edu.au}
\date\today
\thanks{The first author's research was supported by an Australian Government Research Training Program (RTP) Scholarship. The second author was supported by Australian RC Discovery Project DP240100186 and an Australian Mathematical Society Lift-off Fellowship.}
\keywords{}

\begin{abstract}
    Inspired by a classical result of R\'enyi, we prove that every even integer $N\geq 4$ can be written as the sum of a prime and a number with at most 395 prime factors. We also show, under assumption of the generalised Riemann hypothesis, that this result can be improved to 31 prime factors. 
\end{abstract}

\maketitle

\section[1]{Introduction}
In 1948, R\'enyi \cite{Renyi1948} proved the following theorem as an approximation to Goldbach's conjecture.
\begin{theorem}[{\cite[Theorem 1]{Renyi1948}}]
    There exists a natural number $K$ such that every even integer $N\geq 4$ can be written as the sum of a prime and a number with at most $K$ prime factors. 
\end{theorem}
Namely, the case $K=1$ is equivalent to Goldbach's conjecture. Goldbach's conjecture is known to hold for $2<N\leq 4\cdot 10^{18}$ as the result of a large computation by Oliveira e Silva, Herzog and Pardi \cite{O_H_P_14}. For larger values of $N$, Goldbach's conjecture remains an open problem.

If $N$ is sufficiently large, then Chen \cite{Chen1966,Chen1973} proved that one could take $K=2$.
\begin{theorem}[Chen's Theorem]
    Every sufficiently large even integer can be written as the sum of a prime and a number with at most 2 prime factors.
\end{theorem}
There has been little work done however, on determining an explicit value of $K$ that holds for all even $N\geq 4$. One of the reasons for this may be that R\'enyi and Chen's original proofs are ineffective, in that a lower bound for $N$ cannot be determined by following their methods.

Despite this, in \cite{BJV22}, Bordignon and the authors of this paper recently built upon unpublished work of Yamada \cite{Yamada2015} to prove an effective and explicit variant of Chen's Theorem. Namely, they showed \cite[Corollary 4]{BJV22} that Chen's Theorem holds for all even $N\geq\exp(\exp(32.7))$. Using this result, a simple but wasteful argument gives that one can take $K=e^{29.3}\approx 3.2\cdot 10^{13}$ for all $N\geq 4$ \cite[Theorem 5]{BJV22}.

In this paper, by using a more sophisticated procedure that essentially generalises the work in \cite{BJV22}, we improve on this result as follows.
\begin{theorem}\label{unconthm}
    Every even integer $N\geq 4$ can be written as the sum of a prime and a number with at most $K=395$ (not necessarily distinct) prime factors.
\end{theorem}
The main difficulty in lowering the value of $K=395$ comes from our knowledge of potential Siegel zeros and the error term in the prime number theorem for arithmetic progressions for small moduli. As these problems are mitigated under assumption of the Generalised Riemann Hypothesis (GRH), we also provide a conditional result.
\begin{theorem}\label{conthm}
    Assume GRH. Then every even integer $N\geq 4$ can be written as the sum of a prime and a number with at most $K=31$ (not necessarily distinct) prime factors. 
\end{theorem}
It should be noted that obtaining $K=31$ does not require the full-strength of GRH. Rather, if our knowledge of the zeros of Dirichlet $L$-functions were to improve (say with significant computation), then the unconditional result would approach the conditional one. The main hurdle in improving Theorem \ref{conthm} further is the strength of our sieve methods and bounds on sums of the reciprocals of primes (see \S 3). We also remark that our proofs require additional levels of optimisation compared to the explicit version of Chen's theorem in \cite{BJV22}. This is because we consider a wider range of $N$, causing more error terms to become non-negligible. In particular, we make use of recent work by Hathi and the first author \cite{hathi2024sum} for smaller values of $N$.

An outline of the paper is as follows. In Section \ref{sectnot} we provide the main notation and definitions used throughout. In Section \ref{sectpre} we state some preliminary lemmas. In Section \ref{sectuncon} we outline the main method of approach, and prove the unconditional result (Theorem \ref{unconthm}). In Section \ref{sectcon} we prove the conditional result (Theorem \ref{conthm}). Finally in Section \ref{improvsec} we detail possible avenues for future improvements. Note that the main code for the proof of Theorems \ref{unconthm} and \ref{conthm} is available on Github \cite{djvaleriiacode}.

\section{Notation and setup}\label{sectnot}
Here and throughout, $p$ denotes a prime number, $\delta\in(0,2)$, $\alpha\in(0,1/2)$ and $X_2>0$ are parameters we choose later, $N\geq X_2$ is an even integer,
\begin{align*}
    &\gamma=0.57721\ldots\ \text{(Euler's constant)},\\
    &\prod_{p>2}\left(1-\frac{1}{(p-1)^2}\right)=\prod_{p>2}\frac{p(p-2)}{(p-1)^2}=0.66016\ldots\ \text{(Twin prime constant)},\ \text{and}\\
    &x_1=x_1(N):=\frac{N}{\log^5 N}.
\end{align*}
We now state a theorem from \cite{Kadiri2005} to give a precise definition of what it means to be a Siegel zero and an exceptional modulus (see alternatively \cite[Theorem 26]{BJV22}).
\begin{theorem}[{\cite[Theorem 1.1 \& 1.3]{Kadiri2005}}]\label{excepthm}
    Define $\prod (s,q)=\prod_{\chi \pmod q} L(s,\chi)$, $R_0=6.3970$ and $R_1=2.0452$. Then the function $\prod(s,q)$ has at most one zero $\rho=\beta+i \gamma$, in the region $\beta \ge 1-1/\left(R_0\log \max \left( q,q\left|\gamma\right|\right)\right)$. Such a zero is called a Siegel zero and if it exists, then it must be real, simple and correspond to a non-principal real character $\chi \pmod q$. Moreover, for any given $Q$, among all the zeroes of primitive characters with modulus $q \le Q$ there is at most one zero with $\beta \ge 1-1/\left(2R_1 \log Q\right)$, we will call this zero and the related modulus exceptional.
\end{theorem}
Next, we let $Y_0$, $\alpha_1$, $\alpha_2$ and $C = C(\alpha_1,\alpha_2,Y_0)$ be a choice of corresponding values from Table 6 in \cite{Bordignon_21} which are given meaning by the theorem below. The following theorem is a combination of \cite[Theorem 1.2]{Bordignon_21} and \cite[Lemma 27]{BJV22} rewritten in terms of parameters $Y_0$, $\alpha_1$, $\alpha_2$ and $C$.
\begin{theorem}[{\cite[Theorem 1.2]{Bordignon_21} \& \cite[Lemma 27]{BJV22}}]\label{PNTAPthm}
    Let $Y_0$, $\alpha_1$, $\alpha_2$ and $C$ be as in Table 6 of \cite{Bordignon_21} and $X_1={\exp (\exp(Y_0))}$. Let $x>X_1$ and $k<\log^{\alpha_1} x$ be an integer. Let $\textnormal{Ind}_k=1$ if $\beta_k$, the Siegel zero modulo $k$, exists, and $\textnormal{Ind}_k=0$ otherwise. Then for $(k, l) = 1$ we have
    \begin{equation*}
        \frac{\varphi(k)}{x}\left|\psi(x;k,l)-\frac{x}{\varphi(k)}\right|<\frac{C}{\log^{\alpha_2}x}+\textnormal{Ind}_k\frac{x^{\beta_k-1}}{\beta_k},
    \end{equation*}
    where $\psi(x;k,l) = \sum_{n \leq x, n \equiv l\ \mathrm{mod}\ k} \Lambda(n)$, with $\Lambda(n)$ the von Mangoldt function. Moreover,
    \begin{equation*}
        x^{-1}\sum_{\substack{\chi\ \mathrm{(mod\ k)} \\ \chi \ne \chi_0}} \left|\psi(x,\chi)\right|< \frac{C}{\log^{\alpha_2}x}+\textnormal{Ind}_k \frac{x^{\beta_k-1}}{\beta_k},
    \end{equation*}
    where $\chi_0$ denotes the trivial character modulo $k$.
\end{theorem}

Now, let 
\begin{align} \label{eq: K-delta}
    K_\delta(N)&=\log^\delta N,\ \text{and}\\
    Q_{\alpha_1}(N)&=\log^{\alpha_1} N.\label{eq: Q1}   
\end{align}
As at the beginning of Section 4.1 in \cite{BJV22}, we define $k_0:=k_0(N)$ to be the exceptional modulus up to $Q_{\alpha_1}(x_1)$ (if it exists),
\begin{equation*}
    k_1:=k_1(N)=
    \begin{cases}
        k_0,&\text{if $k_0$ exists and $(k_0,N)=1$}\\
        0,&\text{otherwise,}
    \end{cases}
\end{equation*}
and $q_1>\ldots>q_\ell$ be the prime factors of $k_1$ provided $k_1\neq 0$.

We then define $\beta_0=\beta_0(x_1)$ to be an upper bound for all Siegel zeros with modulus $d\leq Q_{\alpha_1}(x_1)$, $d\nmid N$, with the added condition $k_1\nmid d$ if $k_1\geq K_{\delta}(x_1)$. As discussed in the proofs of Lemmas 30 and 33 in \cite{BJV22}, we can set
\begin{equation*}
    \beta_0(x_1)= 1-\nu(x_1)
\end{equation*}
where
\begin{equation}\label{siegelbound}
    \nu(x_1)=\min\left\{\frac{100}{\sqrt{K_{\delta}(x_1)}\log^2 K_{\delta}(x_1)},\frac{1}{2R_1\log(Q_{\alpha_1}(x_1))}\right\}
\end{equation}
if $k_1<K_{\delta}(x_1)$, or
\begin{equation}\label{strongbetabound}
     \nu(x_1)=\frac{1}{2R_1\log(Q_{\alpha_1}(x_1))}
\end{equation}
if $k_1\geq K_{\delta}(x_1)$, with $R_1$ from Theorem \ref{excepthm}. Since there are no Siegel zeros for moduli less than $4\cdot 10^5$ \cite{platt2016numerical}, we can also bound $\beta_0$ by \eqref{strongbetabound} whenever $K_{\delta}(x_1)\leq 4\cdot 10^5$.

\subsection{List of definitions} \label{subsec: list-of-def}
The following list of definitions is adapted directly from \cite{BJV22}. However, we have made small modifications so that everything is expressed in terms of $\alpha_1$, $\alpha_2$, $Y_0$ and $C(\alpha_1,\alpha_2,Y_0)$ rather than the special case $\alpha_1=10$, $\alpha_2=8$, $Y_0=10.4$, $C=3.2\cdot 10^{-8}$ used in \cite{BJV22}. Note that in the subsequent definition of $p_2(X_2)$, and by extension $p(X_2)$, we are assuming the bound \eqref{siegelbound} for $\beta_0(x_1)$. Then, the function $p^*(X_2)$ (which appears in the definition of $c_4^*(X_2)$ below) is equal to $p(X_2)$ but with the stronger bound \eqref{strongbetabound} used for $\beta_0$. Now, without further ado, we define
\begin{align*}
    &A=\{N-p\::\:p\leq N,\ p\nmid N\},\\
    &A_d=\{a \in A,\ d\mid a\},\\
    &S(A,n)=\left|A\backslash\bigcup_{p\mid n}A_p\right|,\\
    &P(z)=\prod_{\substack{p<z\\p\nmid N}}p,\\
    &V(z)=\prod_{p\mid P(z)}\left(1-\frac{1}{p-1}\right),\\
    &U_N=2e^{-\gamma}\prod_{p>2}\left(1-\frac{1}{(p-1)^2}\right)\prod_{\substack{p>2\\p\mid N}}\frac{p-1}{p-2},\\
    &m_j=q_1\ldots q_j,\\
    &P^{(j)}(z)=\prod_{\substack{p<z,\:p\nmid N\\p\neq q_1,\ldots,q_j}}p,\\
    &V^{(j)}(z)=\prod_{p|P^{(j)}(z)} \left(1-\frac{1}{p-1}\right),\\
    &U_N^{(j)}=2e^{-\gamma}\prod_{p>2}\left(1-\frac{1}{(p-1)^2}\right)\prod_{\substack{p>2\\ p\mid Nm_j}}\frac{p-1}{p-2},\\
    &c_{\alpha,X_2,K}=K\left(\frac{1}{2}-\alpha-\frac{2\alpha_1\log\log X_2}{\log X_2}\right), \\
    &r(d)=|A_d|-\frac{|A|}{\varphi(d)},\\
    &h(s)=
    \begin{cases}
        e^{-2},& 1\leq s\leq 2,\\
        e^{-s},& 2\leq s\leq 3,\\
        3s^{-1}e^{-s},& s\geq 3,
    \end{cases}
    \ \\
    &\mathcal{E}(y)=\frac{4y\log^{\frac{9}{2}}y}{\log^{\alpha_1}x_1(y)}+\frac{4y}{\log^{\alpha_1-\frac{9}{2}}y}+\frac{18y^{\frac{11}{12}}}{\log^{\frac{\alpha_1-9}{2}}y}+\frac{5}{2}y^{\frac{5}{6}}\log^{\frac{11}{2}}y\\
    &p_2(X_2)=\max_{y\ge x_1(X_2)}\left[\frac{\log^2 y}{y}\left(1.1\log(Q_{\alpha_1}(x_1))\left(\frac{C(\alpha_1,\alpha_2,Y_0)y}{\log^{\alpha_2}y}+\frac{y^{\beta_0(x_1)}}{\beta_0(x_1)}\right)\right.\right.\\
    &\qquad\qquad\qquad\qquad\qquad\left.\left.+27\cdot\mathcal{E}(y)+\frac{\sqrt{y}}{2(\log 2)\log^{\alpha_1-2}y}+0.4\log^3y\right)\right]
    &\qquad\qquad\qquad\qquad\qquad\qquad+18.78(\log y)^{2.515}\exp(-0.8274\sqrt{\log y})+c_1p_3(y)\Bigg)\Bigg],\\
    &p_1(X_2)=p_2(X_2)+\frac{1}{\log^{\alpha_1-2}x_1(X_2)}\left(0.67+\frac{2}{x_1(X_2)^{\frac{1}{6}}}\right),\\
    &p(X_2)=p_1(X_2)\left(1+\frac{1}{\log^2 X_2\log^3 x_1(X_2)}+\frac{1}{\left(1-\frac{4}{\log x_1(X_2)}\right)\log X_2}\right)+\frac{2.2}{\log^2 X_2}\notag,\\  
    &c(X_2)=c_1(X_2)\left(1+\frac{1}{\log^2(X_2)\log^3x_1(X_2)}+\frac{1}{\left(1-\frac{4}{\log x_1(X_2)}\right)\log X_2}\right)+\frac{1}{\log^2 X_2},\\
    &c_1(X_2) = \max_{y\ge x_1(X_2)}\Bigg[\frac{C(\alpha_1,\alpha_2,Y_0)}{\log^{\alpha_2-2} y} + \log^2y\Bigg(\left(1 - \frac{1}{2 R_1 \log Q_{\alpha_1}(y)} \right)^{-1} y^{- \frac{1}{2 R_1 \log Q_{\alpha_1}(y)}} \\
    &\qquad\qquad\qquad\qquad\quad+ Q_{\alpha_1}(y)\left(\frac{1.02}{\sqrt{y}}+\frac{3}{y^{2/3}}\right)
    + 9.4(\log y)^{1.515}\exp(-0.8274\sqrt{\log y})\Bigg)\Bigg],\\
    &c_2(X_2)=c(X_2)+\frac{1.3841\log^4 X_2}{X_2\log\log X_2},\\
    &c_3(X_2)=\max_{N\ge X_2}\Bigg[\frac{1}{\log\log\log N}\cdot\left(\frac{3}{2\log N}+\frac{\log(N\log^{\alpha_1}x_1(N))}{\log(N/\log^{\alpha_1}x_1(N))}\right)\frac{\log^\delta N}{\log^\delta x_1(N)}\\
    &\qquad\qquad\qquad\qquad\qquad\cdot\left(e^\gamma\log\log\log^\delta x_1(N)+\frac{5}{2\log\log\log^{\delta}x_1(N)}\right)\\
    &\qquad\qquad\qquad\qquad\qquad+\frac{1.3841\log^{2+\delta}N}{N\log\log N\log\log\log N}\Bigg],\\
    &c_4(X_2)=p(X_2)+\frac{0.9\sqrt{x_1(X_2)}\log^{4}X_2}{X_2\log^{\alpha_1}(x_1(X_2))\log\log X_2},\\
    &c_4^*(X_2)=p^*(X_2)+\frac{0.9\sqrt{x_1(X_2)}\log^{4}X_2}{X_2\log^{\alpha_1}(x_1(X_2))\log\log X_2},\qquad\text{(see discussion above)}\\
    &a_1(X_2)=\max_{N\geq X_2}\left[\frac{c_2(X_2)}{\log^{2-\delta}N\log\log\log N}\cdot\frac{1.3841\log(\log^{\alpha_1}x_1(N))}{\log\log(\log^{\alpha_1}x_1(N))}\right]+c_3(X_2),\\
    &a(X_2)=a_1(X_2)\max_{N\geq X_2}\left[\frac{\log\log\log N}{\log^{1+\delta}N}\cdot\prod_{p>2}\frac{(p-1)^2}{p(p-2)}\right.\\
    &\qquad\qquad\qquad\qquad\qquad\left.\cdot\left(e^{\gamma}\log\log(\log^{\alpha_1}x_1(N))+\frac{2.5}{\log\log(\log^{\alpha_1}x_1(N))}\right)\right].
\end{align*}

For our application of the explicit linear sieve in \cite[\S 2]{BJV22} we also need to work with the functions $f(s)$ and $F(s)$ defined by the differential difference equations
\begin{align}
    &F(s)=\frac{2e^{\gamma}}{s},\ f(s)=0,\ 0< s\leq 2, \nonumber\\
    &(sF(s))'=f(s-1),\ (sf(s))'=F(s-1),\ s\geq 2. \label{def-f-F}
\end{align}
From this definition, explicit expressions for $F(s)$ and $f(s)$ can be produced, getting more complicated as $s$ gets larger. In \cite[Lemma 2]{cai2008chen} some of these expressions are listed. For example, for $2\leq s\leq 4$,
\begin{align}\label{fs24}
    f(s)=\frac{2e^{\gamma}\log(s-1)}{s},
\end{align}
for $4\leq s\leq 6$,
\begin{align*}
    f(s)=\frac{2e^{\gamma}}{s}\left(\log(s-1)+\int_3^{s-1}\frac{1}{t}\left(\int_2^{t-1}\frac{\log(u-1)}{u}\mathrm{d}u\right)\mathrm{d}t\right).
\end{align*}
and for $5\leq s\leq 7$,
\begin{align*}
    F(s)&=\frac{2e^{\gamma}}{s}\left(1+\int_2^{s-1}\frac{\log(t-1)}{t}\mathrm{d}t+\int_2^{s-3}\frac{\log(t-1)}{t}\left(\int_{t+2}^{s-1}\frac{1}{u}\log\frac{u-1}{t+1}\mathrm{d}u\right)\mathrm{d}t\right).
\end{align*}

Note that $F(s)$ monotonically decreases towards 1 and $f(s)$ monotonically increases towards 1 \cite[p. 227]{halberstam2011sieve}. Thus, for $s\geq 6$, we can bound $F(s)-1$ and $1-f(s)$ as
\begin{align}
    F(s)-1&\leq F(7)-1\leq 5\cdot 10^{-6},\label{Fbound}\\
    1-f(s)&\leq 1-f(6)\leq 1.05\cdot 10^{-4}.\label{fbound}
\end{align}
Finally, we set 
\begin{align*}
    \overline{m}_{\alpha,X_2,K}:=\max\{(1-f(c_{\alpha,X_2,K}),F(c_{\alpha,X_2,K})-1)\}. 
\end{align*}

\section{Some preliminary results}\label{sectpre}
Here we provide some preliminary lemmas required for sieving, most of which are variants of lemmas from \cite{BJV22}. All notation is as in Section \ref{sectnot}.

First, we show that the inequality $|A|>N/\log N$ holds for a wider range of $N$ than that stated in \cite[Lemma 42]{BJV22}.

\begin{lemma}\label{Alem}
    For all $N\geq 71$,
    \begin{equation*}
        |A|>\frac{N}{\log N}.
    \end{equation*}
\end{lemma}
\begin{proof}
    First consider the case where $N\geq 250$. Note that 
    \begin{equation*}
        |A|=\pi(N)-\omega(N),
    \end{equation*}
    where, as usual, $\pi(N)$ is the number of primes less than or equal to $N$ and $\omega(N)$ is the number of distinct prime factors of $N$. Therefore, we have by \cite[Theorem~2]{rosser1962approximate} and \cite[Th\'eor\`eme~11]{robin1983estimation},
    \begin{equation*}
        |A|\geq \frac{N}{\log N-\frac{1}{2}}-\frac{1.3841\log N}{\log \log N}>\frac{N}{\log N}.
    \end{equation*}
    The case $71\leq N<250$ is then verified by a short computation.
\end{proof}

Next we give two lemmas which are modifications of Lemmas 17 and 18 in \cite{BJV22}.

\begin{lemma}\label{mertenlem}
    We have
    \begin{align}
        \sum_{p<x}\frac{1}{p}&\geq\log\log x+M_0-\frac{2.964\cdot 10^{-6}}{\log x},\quad x\geq 2\label{low1p}\\
        \sum_{p<x}\frac{1}{p}&\leq\log\log x+M_0+1.614\cdot 10^{-3},\quad x>\exp(8.9)\label{up1p1}\\
        \sum_{p<x}\frac{1}{p}&\leq\log\log x+M_0+\frac{6.836\cdot 10^{-6}}{\log x},\quad x>10^{12},\label{up1p2}
    \end{align}
    where $M_0=0.261497\ldots$ is the Meissel--Mertens constant.
\end{lemma}
\begin{proof}
    To begin with, we note that \eqref{low1p} is the same as the lower bound in \cite[Lemma 17]{BJV22}. For \eqref{up1p1} we first used a direct computation for the range $\exp(8.9)\leq x\leq\exp(10)$. Namely, we wrote a short Python script which evaluated each term in the sum $\sum_{p<x}1/p$ successively and compared the result to right-hand side of \eqref{up1p1}. Since $\exp(10)\approx 22000$, this computation was quick and took less than a second on a laptop with a 2.4 GHz processor. Next by \cite[Lemma 16]{BJV22}, for $\exp(10)<x\leq 10^{12}$,
    \begin{equation*}
        \sum_{p\leq x}\frac{1}{p}\leq\log\log x+M_0+\frac{2}{\sqrt{x}\log x} \leq \log \log x + M_0 + 1.348 \cdot 10^{-3}.
    \end{equation*}
    For $x>10^{12}$ it suffices to prove \eqref{up1p2}. By the argument in the proof of \cite[Lemma 17]{BJV22}, we have for $10^{12} < x \leq 10^{19}$,
    \begin{align} \label{eq: bound-1/p}
        \sum_{p\le x}\frac{1}{p}&\leq \log\log x+M_0+\frac{6.9322\cdot 10^{-5}}{\log^2x}+\int_x^{10^{19}}\frac{(y-\theta(y))(1+\log y)}{y^2\log^2y}\mathrm{d}y\notag\\
        &+8.6315\cdot 10^{-7}\left(\frac{\log x}{2\log^2(10^{19})}+\frac{\log x}{\log(10^{19})}\right) \frac{1}{\log x},
    \end{align}
    where
    \begin{equation*}
        \theta(y)=\sum_{p\leq y}\log p
    \end{equation*} 
    is the Chebyshev theta function. We use the bound $x-\theta(x)\leq 1.95\sqrt{x}$ from \cite[Theorem 2]{buthe2018analytic} to obtain
    \begin{align}
        \int_x^{10^{19}}&\frac{(y-\theta(y))(1+\log y)}{y^2\log^2y}\mathrm{d}y\leq\int_x^{10^{19}}\frac{1.95(1+\log y)}{y^{3/2}\log^2y}\mathrm{d}y\notag\\
        &=1.95\left[\frac{1}{2}\text{li}\left(\frac{1}{\sqrt{y}}\right)-\frac{1}{\sqrt{y}\log y}\right]_x^{10^{19}}\notag\\
        &\leq\frac{1}{\log x}\left(\frac{1.95}{\sqrt{x}}-0.975\cdot \text{li}\left(\frac{1}{\sqrt{x}}\right)\log x-2.759\cdot 10^{-11}\cdot\log x\right), \label{eq: bound-int-x-10^19}
    \end{align}
    where 
    \begin{equation*}
        \li(x)=\int_0^x\frac{1}{\log t}\mathrm{d}t
    \end{equation*}
    is the logarithmic integral. We combine \eqref{eq: bound-1/p}, \eqref{eq: bound-int-x-10^19}, and the range $10^{12} < x \leq 10^{19}$ to obtain \eqref{up1p2}. In the case $x > 10^{19}$, we instead have 
    $$\sum_{p\le x}\frac{1}{p}\leq \log\log x+M_0+\frac{6.9322\cdot 10^{-5}}{\log^2x}
    +8.6315\cdot 10^{-7}\left(\frac{1}{2\log^2 x}+\frac{1}{\log x}\right),$$
    implying \eqref{up1p2}.
\end{proof}

\begin{lemma}\label{1pm1lem} Let $u_0\geq 2$ and suppose $z>z_0$. Then for all $u_0<u<z$, we have that there exists an $\epsilon=\epsilon(u_0,z_0)>0$ such that
\begin{equation}\label{mainepseq}
    \prod_{u\leq p<z}\left(1-\frac{1}{p-1}\right)^{-1}<\left(1+\epsilon\right)\frac{\log z}{\log u}.
\end{equation}
For our purposes, we have that
\begin{equation}\label{u30eq}
    \epsilon(30,\exp(8.9))=1.312\cdot 10^{-2}
\end{equation}
and
\begin{equation}\label{u400eq}
    \epsilon(400,10^{12})=5.543\cdot 10^{-4}
\end{equation}
are valid choices of $\epsilon$.
\end{lemma}
\begin{proof}
    We modify the proof of \cite[Lemma 18]{BJV22} but bound some of the terms with more care since we are dealing with much lower values of $u$. We will first prove \eqref{u30eq}, noting that the proof for \eqref{u400eq} is identical with slight changes in the constants obtained. In general, the following argument can be used to obtain \eqref{mainepseq} for any choice of $u_0$ and $z_0$. So to begin with,
    \begin{equation*}
        \prod_{u\le p<z}\left(1-\frac{1}{p-1}\right)^{-1}=\prod_{u\le p<z}\left(\frac{(p-1)^2}{p(p-2)}\right)\prod_{u\le p<z}\left(1-\frac{1}{p}\right)^{-1}.
    \end{equation*}
    Now,
    \begin{align}
        \prod_{u\le p<z}\left(\frac{(p-1)^2}{p(p-2)}\right)&\leq \prod_{p\geq u}\left(\frac{(p-1)^2}{p(p-2)}\right)\notag\\
        &=\frac{\prod_{p>2}\left(\frac{(p-1)^2}{p(p-2)}\right)}{\prod_{2<p<u}\left(\frac{(p-1)^2}{p(p-2)}\right)}\notag\\
        &\le 1.00754,\label{twinprimerat}
    \end{align}
    noting that $u>u_0=30$ and $\prod_{p>2}\left(\frac{(p-1)^2}{p(p-2)}\right)=1.514780\ldots$ is the reciprocal of the twin prime constant. Thus,
    \begin{equation}\label{prodeq1}
        \prod_{u\le p<z}\left(1-\frac{1}{p-1}\right)^{-1}<1.00754\prod_{u\le p<z}\left(1-\frac{1}{p}\right)^{-1}.
    \end{equation}
    Next, we note that
    \begin{equation}\label{prodtoexp}
        \prod_{u\le p<z}\left(1-\frac{1}{p}\right)^{-1}=\exp\left(-\sum_{u\le p<z}\log\left(1-\frac{1}{p}\right)\right).
    \end{equation}
    Now, by Lemma \ref{mertenlem}
    \begin{align}\label{merteneq}
        \sum_{u\le p< z}\frac{1}{p}&=\sum_{p<z}\frac{1}{p}-\sum_{p<u}\frac{1}{p}\notag\\
        &\le\log\log z-\log\log u+1.614\cdot 10^{-3}+\frac{2.964\cdot 10^{-6}}{\log u}\notag\\
        &\le \log\log z-\log\log u + 1.615\cdot 10^{-3}.
    \end{align}
    since $z>\exp(8.9)$ and $u>30$.
    Using \eqref{prodtoexp}, \eqref{merteneq} and that for $0<x\le 1/10$,
    \begin{equation*}
        \log(1-x)\ge-x-0.54x^2,
    \end{equation*}
    we have,
    \begin{equation}\label{prodeq2}
        \prod_{u\le p<z}\left(1-\frac{1}{p}\right)^{-1}\le\frac{\log z}{\log u}\exp(1.615\cdot 10^{-3})\exp\left(0.54\sum_{p\ge u}\frac{1}{p^2}\right)
    \end{equation}
    with
    \begin{equation}\label{prodeq4}
        0.54\sum_{p\ge u}\frac{1}{p^2}\leq 0.54\sum_{p\geq 2}\frac{1}{p^2}-0.54\sum_{p<u}\frac{1}{p^2}< 0.00390788
    \end{equation}
    since $u>30$ and $\sum_{p}\frac{1}{p^2}=0.45224742004106\ldots$ is known to a high degree of accuracy (see e.g.\ \cite{merrifield1882iii}). Using \eqref{prodeq1}, \eqref{prodeq2} and \eqref{prodeq4} then gives \eqref{u30eq}. 
    
    For the second result \eqref{u400eq} the method is exactly the same. However, the constants appearing in \eqref{twinprimerat}, \eqref{prodeq2} and \eqref{prodeq4} are instead replaced by $1.00036$, $7.422\cdot 10^{-7}$ and $0.000193433$ respectively.  
\end{proof}

We now recall \cite[Lemma 38]{BJV22}, giving additional examples that we require.

\begin{lemma}[{\cite[Lemma 38]{BJV22}}]\label{vjlem}
    For $x\ge 285$ and $j=0,\ldots,\ell$, we have
    \begin{align} \label{eq: Vj-UNj}
        V^{(j)}(x) &= \frac{U_N^{(j)}}{\log x} \left[ 1 + 1.45~\theta_1(x) \frac{\log N}{x-1} \left( 1 + \frac{10 \log \log N}{\log N} \right) \right]\notag\\
        &\qquad\qquad\qquad\qquad\qquad\qquad\cdot\left( 1 + \frac{1.002~\theta_2(x)}{x-3} \right) \left( 1 + \frac{\theta_3(x)}{2 \log^2 x} \right),
    \end{align}
    where $|\theta_i(x)|\leq 1$. In particular, for a choice of positive integer $M$ we set $z=N^{1/M}$ and $z\geq z_0$, allowing us to write
    \begin{equation} \label{def: xi}
        \frac{U_N^{(j)}}{\log z}\left(1-\frac{\xi(z_0,M)}{\log^2 N}\right)<V^{(j)}(z)<\frac{U_N^{(j)}}{\log z}\left(1+\frac{\xi(z_0,M)}{\log^2 N}\right)
    \end{equation}
    for some constant $\xi(z_0,M)>0$. For our purposes, we compute $\xi(10^{12},40)\leq 801$ and $\xi(\exp(8.9),18))\leq 1244$.
\end{lemma}

Finally we give a result that follows directly from \cite[Theorem 1.5]{hathi2024sum}. This improves on an earlier result of Dudek \cite{dudek2017sum} that was used in \cite{BJV22}.

\begin{lemma}\label{djshehzadlem}
    Let $p_i$ denote the $i^{\text{th}}$ prime and suppose $X_2\geq 4\cdot 10^{18}$. Then every even integer $2<N<X_2$ can be written as the sum of a prime and a square-free number $\eta>1$ with at most $K$ prime factors, where $K\geq 1$ is the largest integer such that
    \begin{equation*}
        \theta(p_{K+6})-\theta(13)<\log(X_2).
    \end{equation*}
\end{lemma}
\begin{proof}
    For $2<N\leq 4\cdot 10^{18}$, the result is true since Goldbach's conjecture holds in this range \cite{O_H_P_14}. For $4\cdot 10^{18}<N<X_2$ we then have by \cite[Theorem 1.5]{hathi2024sum} that $N=p+\eta$ where $p$ is a prime and $\eta$ is a square-free number coprime to the first 6 primes 2, 3, 5, 7, 11 and 13. Since $\eta<N\leq X_2$, the number of prime factors of $\eta$ is at most
    \begin{equation*}
        K=\max_m\left\{\prod_{i=1}^mp_{i+6}<X_2\right\}
    \end{equation*}
    and if $\prod_{i=1}^Kp_{i+6}<X_2$ then $\theta(p_{K+6})-\theta(13)=\sum_{i=1}^K\log(p_{i+6})<\log(X_2)$.
\end{proof}
\begin{remark}
    The condition $X_2\geq 4\cdot 10^{18}$ can be weakened to $X_2\geq 40$. However, here we wish to highlight the usefulness of the Goldbach verification \cite{O_H_P_14}.
\end{remark}

\section{The unconditional result}\label{sectuncon}
In this section, we prove Theorem \ref{unconthm}. Namely, that every even integer $N\geq 4$ can be expressed as the sum of a prime and a number with at most $K=395$ prime factors. In what follows we set $\pi_K(N)$ to be the number of ways to write $N$ as the sum of a prime and a number with at most $K$ prime factors. The general idea will be to set $z=N^{1/M}$ for some positive integer $M$ satisfying $5\leq M\leq K+1$. We then have $\pi_{M-1}(N)\geq S(A,P(z))$ so that, if one can prove $S(A,P(z))>0$ for all even $N\geq X_2$, then $\pi_{K}(N)\geq\pi_{M-1}(N)>0$ for all even $N\geq X_2$. Since we will be taking $X_2$ to be quite large, the case when $4\leq N<X_2$ must be treated separately. This will be done using Lemma \ref{djshehzadlem}.

To bound $S(A,P(z))$ from below, we generalise Theorem 44 of \cite{BJV22}. This is done by parameterising\footnote{To avoid confusion with notation, we remark that $\alpha_1$ means something different in \cite{BJV22}. Namely, it corresponds to our variable $\alpha$.} $\alpha_1$, $\alpha_2$, $Y_0$ and $M$ and making some other small changes. Similar to \cite{BJV22}, the method will be to pick suitable values of the parameters $\alpha_1$, $\alpha_2$, $Y_0$, $\delta$, $\alpha$, $M$ and $X_2$ so that $S(A,P(z))>0$ for all even $N\geq X_2$. 

We first provide a modification of \cite[Lemma 43]{BJV22}, which will be used in the proof of Theorem \ref{theo:S>1} below.

\begin{figure}
\centering
\begin{tikzpicture}
\node[draw] at (0,0) (L43) {Lemma \ref{Lemma42-BJV}};
\node[draw] at (-5,-1) (L22) {\cite[Lemma 22]{BJV22}};
\node[draw] at (-2,-3) (L31) {\cite[Lemma 31]{BJV22}};
\node at (-2,-3.6) (L31-1) {\small{with $p(X_2)$ as in Sec. \ref{subsec: list-of-def}}};
\node at (-2,-4.1) (L31-2) {\small{and $\log \log x_1(X_2) \geq 10.4$}};
\node at (-2,-4.6) (L31-3) {\small{\quad\quad replaced by $\ \geq Y_0$}};
\node[draw] at (2,-3) (L33) {\cite[Lemma 33]{BJV22}};
\node at (2,-3.6) (L33-1) {\small{$\log \log x_1(X_2) \geq 10.4$}};
\node at (2,-4.1) (L33-2) {\small{replaced by $\ \geq Y_0$}};
\node[draw] at (5,-1) (L34) {\cite[Lemma 34]{BJV22}};
\node at (5,-1.6) (L34-1) {\small{with $p^*(X_2)$ as in Sec. \ref{subsec: list-of-def}}};
\node at (5,-2.1) (L34-2) {\small{and $\log \log x_1(X_2) \geq 10.4$}};
\node at (5,-2.6) (L34-3) {\small{\quad\quad replaced by $\ \geq Y_0$}};
\node[draw] at (-2,-6) (L30) {\cite[Lemma 30]{BJV22}};
\node at (-2,-6.6) (L30-1) {\small{with $H := \frac{\sqrt{x_1}}{\log^{\alpha_1} x_1}$ and}};
\node at (-2,-7.1) (L30-2) {\small{$\log \log x_1(X_2) \geq 10.4$}};
\node at (-2,-7.6) (L30-3) {\small{replaced by $\ \geq Y_0$}};
\node[draw] at (2,-6) (L27) {\cite[Lemma 27]{BJV22}};
\node at (2,-6.6) (L27-1) {\small{replaced by}};
\node at (2,-7.1) (L27-2) {\small{Theorem \ref{PNTAPthm}}};
\node[draw] at (-5,-9.5) (L23) {\cite[Lemma 23]{BJV22}};
\node at (-5,-10.1) (L23-1) {\small{requires}};
\node at (-5,-10.6) (L23-2) {\small{$45 \leq H$}};
\node[draw] at (-1.7,-9.5) (L24) {\cite[Lemma 24]{BJV22}};
\node at (-1.7,-10.1) (L24-1) {\small{requires}};
\node at (-1.7,-10.6) (L24-2) {\small{$10^9 \leq Q_{\alpha_1}(x_1) \leq H$}};
\node[draw] at (1.6,-9.5) (Hincr) {$H$ increases in $N$};
\node[draw] at (4.9,-9.5) (L28) {\cite[Lemma 28]{BJV22}};
\node at (4.9,-10.1) (L28-1) {\small{replaced by}};
\node at (4.9,-10.6) (L28-2) {\small{Theorem \ref{PNTAPthm}}};

\draw [->] (L22.south) to [out=-90,in=180] (L30.west);
\draw [->] (L22.north) to [out=90,in=180] (L43.west);
\draw [->] (L31.north) to [out=90,in=-90] (-0.1, -0.3);
\draw [->] (L33.north) to [out=90,in=-90] (0.1, -0.3);
\draw [->] (L34.north) to [out=90,in=0] (L43.east);
\draw [->] (L30.north) to [out=90,in=-90] (L31-3.south);
\draw [->] (L27.north) to [out=90,in=-90] (L33-2.south);
\draw [->] (L23.north) to [out=90,in=180] (L31.west);
\draw [->] (L24.north) to [out=90,in=-90] (-2,-7.9);
\draw [->] (Hincr.north) to [out=150,in=-60] (-1.5,-7.9);
\draw [->] (L28.north) to [out=150,in=-30] (-1,-7.9);
\end{tikzpicture}
\caption{The scheme of the proof of Lemma \ref{Lemma42-BJV}. Each node indicates which lemma we need to use from \cite{BJV22} and how these lemmas must be updated for our purposes. An arrow $A\to B$ between two nodes indicates that the lemma $A$ is required to prove the lemma $B$.}
\label{pic: Lemma 4.1}
\end{figure}

\begin{lemma}\label{Lemma42-BJV}
    Let $N$ be a positive even integer with $N \geq X_2$ and $\log\log x_1(X_2)~\ge~Y_0$. Let $H:=H(N)=\frac{\sqrt{x_1}}{\log^{\alpha_1}x_1}$ and suppose $K_{\delta}$, $c_4(X_2)$, $c_4^*(X_2)$, $P^{(j)}$, and $Q_{\alpha_1}(x_1)$ are as in Section \ref{sectnot}. If $10^9 \leq Q_{\alpha_1}(x_1) \leq H$ and $H = H(N)$ increases in $N$, then
    \begin{equation}\label{unsharpeq1}
        \sum_{\substack{d<H\\d\mid P(z)}}|r(d)|<\frac{c_4(X_2)N}{\log^3 N}
    \end{equation}
    and if $k_1\ge K_{\delta}(x_1)$
    \begin{equation}\label{sharpeq}
         \sideset{}{^\sharp}\sum_{\substack{d<H/m_j}}|r_{m_j}(d)|<\frac{c_4^*(X_2)N}{\log^3N},
    \end{equation}
    for all $1 \leq j \leq \ell$, where
    and the $\sharp$ means that the sum is over $d\mid P^{(j+1)}(z)$ if $j<\ell$ and $d\mid P^{(\ell)}(z)$ if $j=\ell$.
\end{lemma}
\begin{proof}
The proof is the same as for \cite[Lemma 43]{BJV22} with a few modifications in the required conditions and subsidary lemmas. We provide a scheme, Figure \ref{pic: Lemma 4.1}, to illustrate how the lemmas from \cite{BJV22} are combined for the proof \cite[Lemma 43]{BJV22}. We also state which conditions in each of these lemmas should be updated in order to get Lemma \ref{Lemma42-BJV} instead of \cite[Lemma 43]{BJV22}. In particular, most of the arguments leading to the proof of \cite[Lemma 43]{BJV22} did not actually require $\log\log x_1(X_2)\geq 10.4$ and can readily be replaced by $\log\log x_1(X_2)\geq Y_0$. 
\end{proof}

\begin{theorem}
\label{theo:S>1}
Let $K_{\delta}$ and the functions $f$, $F$ be defined in \eqref{eq: K-delta} and \eqref{def-f-F} respectively.
Let $\alpha_1$, $\alpha_2$, $Y_0$ and $C(\alpha_1,\alpha_2,Y_0)$ be a valid choice of parameters in Theorem \ref{PNTAPthm}, $M\geq 5$, $u_0\geq 2$, $\delta\in(0,2)$, $\alpha\in(0,1/2)$ and $X_2>0$. Also let $N\geq X_2$ be even, $z=N^{1/M}$, $z_0=X_2^{1/M}$, $\epsilon=\epsilon(u_0,z_0)$ be as in Lemma \ref{1pm1lem}, and $\xi(z_0,M)$ be as in Lemma \ref{vjlem}. Assume $\log\log x_1(X_2)\ge Y_0$ and the conditions
\begin{align}\label{newconditions}
   &\frac{\sqrt{x_1}}{\log^{\alpha_1}x_1}\geq \log^{\alpha_1}N\geq 10^9,\quad 1-\frac{\xi(z_0,M)}{\log^2N}\geq 0,\quad X_2\geq 4\cdot 10^{18},\quad\epsilon\leq \frac{1}{74}\\
   &\frac{N^{\alpha}}{\log^{\alpha_1} x_1(N) \log^{2.5} N}\ge \exp\left(u_0\left(1+\frac{9\cdot 10^{-7}}{\log u_0} \right)\right),~
    \frac{N^{\frac{1}{2}-\alpha}}{\log^{2\alpha_1} N}\ge z^2,\ K_{\delta}(x_1)\geq 3022.\notag
\end{align}
are satisfied.\\
(a) If $k_1<K_{\delta}(x_1)$, we have
\begin{align*}
    S(A,P(z))&>M\frac{|A| U_N}{\log N}\left(1-\frac{\xi(z_0,M)}{\log^2 N}\right)\\
    &\cdot\Bigg(f\left(M\left(\frac{1}{2}-\alpha\right)\right)-C_2(\epsilon)\epsilon e^2h\left(M\left(\frac{1}{2}-\alpha\right)\right)\\
    &-\frac{1}{M}\left(1-\frac{\xi(z_0,M)}{\log^2 N}\right)^{-1}\left(2e^{-\gamma}\prod_{p>2}\left(1-\frac{1}{(p-1)^2}\right)\right)^{-1}\frac{c_4(X_2)}{\log N}\Bigg).
\end{align*}
(b) On the other hand, if $k_1\ge K_{\delta}(x_1)$, we have
\begin{align*}
    &S(A,P(z))>\\
    &\ M\frac{|A| U_N}{\log N}\left(1+\frac{\xi(z_0,M)}{\log^2N}\right)\Bigg\{f(c_{\alpha,X_2,K})-\epsilon_1(X_2,\delta)(1-f(c_{\alpha,X_2,K})) \\
    &\ -(1+\epsilon_1(X_2,\delta))\epsilon C_2(\epsilon)e^2h(c_{\alpha,X_2,K}) \\
    &\ -\left(3\epsilon_1(X_2,\delta)+a(X_2)\right)\cdot(\overline{m}_{\alpha,X_2,K}+\epsilon C_1(\epsilon)e^2h(c_{\alpha,X_2,K}))-a(X_2)-\frac{2\xi(z_0,M)}{\log^2N}\\
    &\ -\frac{1}{M}\left(1+\frac{\xi(z_0,M)}{\log^2 N}\right)^{-1}\left(2e^{-\gamma}\prod_{p>2}\left(1-\frac{1}{(p-1)^2}\right)\right)^{-1}\frac{c_4^*(X_2)}{\log N}\frac{1.3841\log(\log^{\alpha_1}x_1(N))}{\log\log(\log^{\alpha_1}x_1(N))}\Bigg\},
\end{align*}
where $C_1(\epsilon)$ and $C_2(\epsilon)$ are the values in \cite[Table 1]{BJV22}, $\epsilon_1(X_2,\delta)=\frac{1}{\overline{p}-2}$ with $\overline{p}$ the largest prime such that
\begin{equation*}
    K_{\delta}(x_1(X_2))\geq\prod_{2<p\leq \overline{p}}p
\end{equation*}
and all other notation is as in Section \ref{sectnot}.
\end{theorem}
\begin{remark}
    The restriction $M\geq 5$ is so that the condition $\frac{N^{1/2-\alpha}}{\log^{2\alpha_1}N}\geq z^2$ can be satisfied. This also means that our approach works for at best $K=4$ prime factors.
\end{remark}
\begin{proof}[Proof of Theorem \ref{theo:S>1}]
The proof is very similar to that of \cite[Theorem 44]{BJV22}, except with more choices of parameters and corresponding changes to the conditions. As a result we will only outline the proof of the theorem. However, let us first comment on these changes before we provide such an outline. The first of the seven initial conditions in \eqref{newconditions} is required for applying Lemma \ref{Lemma42-BJV}, and the second is required for applying Lemma \ref{vjlem}; these conditions were true for $Y_0 = 10.4$, $M = 8$, ${z_0 = \exp(20)}$, and $\xi(\exp(20),8) = 32.02$ in \cite{BJV22} but are not necessarily satisfied in our more generalised setting. Note that we removed the condition $8\alpha_1+\frac{160\log \log N}{\log N}<1$, assumed in \cite[Theorem 44]{BJV22}, as this was only required to give an exact expression for $f(s)$ and ensure that the lower bound for $S(A,P(z))$ was asymptotically large enough to prove Chen's theorem.

Let us now outline the proof of the theorem in case (a). We first note that the parameter $\alpha$ from the statement of Theorem \ref{theo:S>1} coincides with the parameter $\alpha_1$ from \cite[Theorem 44]{BJV22}. 

Let
\begin{equation} \label{def: D-s-Q}
    D=N^{\frac{1}{2}-\alpha},\ s=\frac{\log D}{\log z}=M\left(\frac{1}{2}-\alpha\right)\quad\text{and}\quad Q=\prod_{\substack{p\le u_0\\p\nmid N}}p.
\end{equation}
Here, $D \geq z^2$ by the fifth of the seven conditions \eqref{newconditions}, and thus by \cite[Theorem 6]{BJV22} and \eqref{def: xi}
\begin{equation} \label{eq: S>-uncond-case-a}
    S(A,P(z))>M\frac{|A| U_N}{\log N}\left(1-\frac{\xi(z_0,M)}{\log^2 N}\right)
    \cdot\left(f\left(s\right)-C_2(\epsilon)\epsilon e^2h\left(s\right)\right)- \sum_{\substack{d\mid P(z)\\ d<QD}}|r(d)|.
\end{equation}

Next, the fifth condition from \eqref{newconditions} and \cite[Lemma 25]{BJV22} imply $QD \leq H := \frac{\sqrt{x_1}}{\log^{\alpha_1} x_1}$, and thus we can apply Lemma \ref{Lemma42-BJV} to bound the second term on the right-hand side of \eqref{eq: S>-uncond-case-a}. Following the proof of \cite[Theorem 44, case (a)]{BJV22} \textit{mutatis mutandis} then allows us to complete the proof of case (a).

For case (b), the proof also follows from the same method as that of \cite[Theorem 44, case (b)]{BJV22}. We provide a scheme, Figure \ref{pic: Theorem 4.2}, which shows the changes one needs to make to the subsidiary lemmas in \cite{BJV22} in order to prove our more general result. We note that \cite[Lemma 42]{BJV22} requires the condition $\log\log(x_1(X_2)) \geq 10.4$ to bound $|A|$ from above (see \cite[(125)]{BJV22}). We don't require this upper bound for $|A|$ to prove Theorem \ref{theo:S>1} but only the lower bound, which follows from Lemma \ref{Alem}.

\begin{figure}
\centering
\begin{tikzpicture}
\node[draw] at (0,0) (T44) {Theorem \ref{theo:S>1}, case (b)};
\node[draw] at (-5,-1) (L38) {\cite[Lemma 38, (114)]{BJV22}};
\node at (-5,-1.6) (L22-1) {\small{replaced by}};
\node at (-5,-2.1) (L22-2) {\small{Lemma \ref{vjlem} eq. \eqref{def: xi}}};
\node[draw] at (-5,-3.4) (L39) {\cite[Lemma 39]{BJV22}};
\node at (-5,-4) (L39-1) {\small{with $10$ replaced by $\alpha_1$}};
\node[draw] at (-5,-5.8) (L42) {\cite[Lemma 42]{BJV22}};
\node at (-5,-6.4) (L42-1) {\small{with $c_2(X_2)$ and}};
\node at (-5,-6.9) (L42-2) {\small{$c_3(X_2)$ as in Sec. \ref{subsec: list-of-def}}};

\node[draw] at (0,-1.5) (L41) {\cite[Lemma 41]{BJV22}};
\node[draw] at (0,-3) (L45) {\cite[Lemma 45]{BJV22}};
\node[draw] at (0,-4.5) (L46) {\cite[Lemma 46]{BJV22}};
\node at (0,-5.1) (L46-1) {\small{requires $\frac{N^{1/2 - \alpha_1}}{\log^{2\alpha_1}x_1} \geq z^2$}};

\node[draw] at (5,-1) (L49) {\cite[Lemmas 48-49]{BJV22}};
\node at (5,-1.6) (L49-1) {\small{with $a(X_2)$}};
\node at (5,-2.1) (L49-2) {\small{as in Sec. \ref{subsec: list-of-def}}};
\node[draw] at (5,-3.4) (L47) {\cite[Lemma 47]{BJV22}};
\node[draw] at (5,-5.8) (L43) {Lemma \ref{Lemma42-BJV}};

\draw [->] (L38.north) to [out=90,in=180] (T44.west);
\draw [->] (L49.north) to [out=90,in=0] (T44.east);
\draw [->] (L41.north) to [out=90,in=-90] (T44.south);
\draw [->] (L39.east) to [out=0,in=-90] (-1.6, -0.3);
\draw [->] (L47.west) to [out=180,in=-90] (1.6, -0.3);
\draw [->] (L42.east) to [out=0,in=-90] (-1.4, -0.3);
\draw [->] (L43.north) to [out=90,in=-90] (L47.south);
\draw [->] (L45.north east) to [out=45,in=-80] (1.1, -0.3);
\draw [->] (L46.north east) to [out=45,in=-85] (1.4, -0.3);
\end{tikzpicture}
\caption{The scheme of the proof of Theorem \ref{theo:S>1} in case (b).}
\label{pic: Theorem 4.2}
\end{figure}
Finally, we remark that the condition $1-\xi(z_0,M)/\log^2N\geq 0$ is required to prevent any sign problems when applying Lemma \ref{vjlem}, $\epsilon\leq 1/74$ is required to apply \cite[Theorem 6]{BJV22}, and the condition $X_2\geq 4\cdot 10^{18}$ is chosen as for $N\leq 4\cdot 10^{18}$ we always have $\pi_{M-1}(N)>0$ by \cite{O_H_P_14}. Certainly, some these conditions can be weakened if desired, but they are easily satisfied in all the scenarios we consider.
\end{proof}


\begin{proof}[Proof of Theorem \ref{unconthm}]
    From \cite[Table 6]{Bordignon_21} we have that
    \begin{align}
        (Y_0,\alpha_1,\alpha_2,C)&=(7.9, 7, 2, 3.98)\label{param2}
    \end{align} 
    are valid choices of parameters. 
    
    We choose suitable values for $u_0$, $\delta$, $\alpha$, $M$ and $X_2$ to use in Theorem \ref{theo:S>1}. Through a process of trial and error, we found that $u_0=400$, $\delta=1.3$, $\alpha=0.25$ and $M=40$ gave $S(A,P(z))>0$ for all even $N\geq X_2=\exp(\exp(7.9))$. Other values of $\delta$, $\alpha$ and $M$ that we tested only worked for equal or larger values of $X_2$. One could optimise these parameters to more decimal places if desired. However, we found that the main bottleneck to further improvement was the values of $Y_0$, $\alpha_1$, $\alpha_2$ and $C$ that we obtained from \cite{Bordignon_21}.
    
    Note that with these choices of parameters we have $\epsilon=5.543\cdot 10^{-4}$ (Lemma \ref{1pm1lem}), $C_1(\epsilon)=113$, $C_2(\epsilon)=114$, $\xi(z_0,M)\leq 801$ (Lemma \ref{vjlem}), $c_{\alpha,X_2,K}\geq 8.23$, $1-f(c_{\alpha,X_2,K})\leq 1.05\cdot 10^{-4}$ (from \eqref{fbound}), $1-f(M(1/2-\alpha))\leq 1.05\cdot 10^{-4}$, ${h(M(1/2-\alpha))\leq 1.4\cdot 10^{-5}}$, $h(c_{\alpha_{X_2}})\leq 10^{-5}$ and $\epsilon_1(X_2,\delta)=1/11$. More explicitly, Theorem \ref{theo:S>1} gives
    \begin{align*}
        S(A,P(z))&>\frac{38U_NN}{\log^2 N},\quad k_1< K_{\delta}(x_1),\\
        S(A,P(z))&>\frac{5 U_NN}{\log^2 N},\quad k_1\geq K_{\delta}(x_1)
    \end{align*} 
    for all even $N\geq X_2$ with $z=N^{1/40}$. Here, we have used that $|A|>N/\log N$ (Lemma \ref{Alem}) after first verifying that $S(A,P(z))>0$. This tells us that every even integer $N\geq X_2$ can be written as the sum of a prime and a number with at most $M-1=39$ prime factors. For the range $2<N<X_2$ we then apply Lemma \ref{djshehzadlem} and obtain the final value $K=395$.
\end{proof}
\begin{remark}
    Although it may seem that taking $M$ larger than 40 would lead to a better result, this is not necessarily the case. In particular, as $M$ gets larger, so does $\xi(z_0,M)$ to the point where it negatively affects the second condition in \eqref{newconditions} and the bounds on $S(A,P(z))$. Moreover, as $M$ gets large, $z=N^{1/M}$ decreases and worse bounds must be used in Lemma \ref{1pm1lem}.
\end{remark}

\section{The conditional result}\label{sectcon}
In this section we prove Theorem \ref{conthm}. As assuming GRH allows for many improvements to the unconditional result, this section is quite large and has been split into three parts. To begin with, we will use some recent results of Greni\'e and Molteni \cite{ExplChebotarev2019} to obtain conditional bounds for the error terms $\left|\pi(x;q,a)-\frac{\li(x)-\li(2)}{\varphi(q)}\right|$ and $\left|\theta(x;q,a)-\frac{x}{\varphi(q)}\right|$ appearing in the prime number theorem for arithmetic progressions. Next, we will extend Lemma \ref{djshehzadlem} under assumption of GRH. Finally, we will prove a conditional lower bound on $S(A,P(z))$ and use this to prove Theorem \ref{conthm}.

We note that in Sections \ref{condboundssect} and \ref{sapsubsect} there are some similarities with recent work due to Bordignon and the second author \cite{bordignon2024explicit}, which was written concurrently with this paper. However, we have still included all the details here to make this paper self-contained.

\subsection{Conditional bounds on $\pi(x;q,a)$ and $\theta(x;q,a)$}\label{condboundssect}
First we give a bound on $\left|\pi(x;q,a)-\frac{\li(x)}{\varphi(q)}\right|$ which will later be used in Section \ref{sapsubsect} as part of the lower bound on $S(A,P(z))$.
\begin{lemma}\label{epilem}
    Let $x\geq X_2\geq 4\cdot 10^{18}$, and $q$ and $a$ be integers such that $3\leq q\leq \sqrt{x}$ and $(a,q)=1$. Then, assuming GRH,
    \begin{equation*}
        \left|\pi(x;q,a)-\frac{\li(x)}{\varphi(q)}\right|\leq c_{\pi}(X_2)\sqrt{x}\log x,
    \end{equation*}
    where 
    \begin{align*}
        c_{\pi}(X_2)&=\frac{3}{8\pi}+\frac{6 + 1/\pi}{4\log X_2}+\frac{6}{\log^2 X_2} + \frac{\li(2)}{\sqrt{X_2} \log X_2} \leq 0.16.
    \end{align*}
\end{lemma}
\begin{proof}
We apply \cite[Corollary 1]{ExplChebotarev2019} for $K = \mathbb{Q}$ and $L = \mathbb{Q}\left[e^{2\pi i/q}\right]$ and the bounds on $x$ and $q$ to get
\begin{equation*}
    \left|\pi(x;q,a)-\frac{\li(x)-\li(2)}{\varphi(q)}\right|\leq \left( \frac{3}{8\pi}+\frac{6 + 1/\pi}{4\log X_2}+\frac{6}{\log^2 X_2} \right) \sqrt{x}\log x.
\end{equation*}
We conclude by using the triangle inequality and $\varphi(q) \geq 1$.
\end{proof}

We now provide a similar style result for $\left|\theta(x;q,a)-\frac{x}{\varphi(q)}\right|$ which will be useful in Section \ref{djshehzadsect}.

\begin{lemma}\label{cthetalem}
     Let $x\geq X_3\geq 4\cdot 10^{18}$, and $q$ and $a$ be integers such that $1\leq q\leq x$ and $(a,q)=1$. Then, assuming GRH,
    \begin{equation*}
        \left|\theta(x;q,a)-\frac{x}{\varphi(q)}\right|< c_{\theta}(X_3)\sqrt{x}\log^2x,
    \end{equation*}
    where 
    \begin{align}\label{cthetaeq}
        c_{\theta}(X_3)&=\frac{5}{8\pi}+\frac{2}{\log X_3}+\frac{(3 + 1.93378 \cdot 10^{-8})}{\log^2X_3} + \frac{1.04320}{X_3^{1/6}\log^2 X_3}\\
        &\leq 0.25.\notag
    \end{align}
\end{lemma}
\begin{proof}
    
    We first obtain bounds for $\left|\psi(x;q,a)-\frac{x}{\varphi(q)}\right|$. By $1\leq q\leq x$ and \cite[Theorem 1]{ExplChebotarev2019} applied to $K = \mathbb{Q}$ and $L = \mathbb{Q}\left[e^{2\pi i/q}\right]$, we get
    \begin{align*}
        \left| \psi(x;q,a) - \frac{x}{\varphi(q)} \right| &< c_{\psi}(X_3)\sqrt{x}\log^2 x,
    \end{align*}
    where
    \begin{align*}
        c_{\psi}(X_3) &=\frac{5}{8\pi}+\frac{2}{\log X_3}+\frac{2}{\log^2 X_3}.
    \end{align*}
    Next, by \cite[Corollary 5.1]{BKLNW_21}, we have for all $x \geq X_3\geq 4\cdot 10^{18}$,
    \begin{equation*}
        \psi(x;q,a)-\theta(x;q,a)\leq\psi(x) - \theta(x) \leq (1 + 1.93378 \cdot 10^{-8}) \sqrt{x} + 1.04320 x^{1/3}.
    \end{equation*}
    Hence,
    \begin{align*}
        \left| \theta(x;q,a) - \frac{x}{\varphi(q)} \right| \leq c_{\theta}(X_3) \sqrt{x} \log^2 x,
    \end{align*}
    with $c_{\theta}(X_3)$ as in \eqref{cthetaeq}.
\end{proof}

\subsection{An extension of Lemma \ref{djshehzadlem}}\label{djshehzadsect}
Lemma \ref{djshehzadlem} is based off a result of Hathi and the first author \cite[Theorem 1.5]{hathi2024sum} which gives that any even integer $N\geq 40$ can be expressed as the sum of a prime and a square-free number that is coprime to the primorial $2\cdot 3\cdot 5\cdot 7\cdot 11\cdot 13=30030$. As this result is quite useful in obtaining our final value for $K$, here we provide an extension which is conditional under GRH and involves larger primorials and values of $N$.

We begin with a variation of \cite[Lemma 5.1]{hathi2024sum} and then bound some terms for ease of computation.

\begin{lemma}[{cf. \cite[Lemma 5.1]{hathi2024sum}}]\label{hathilem}
Define $\overline{R}_k(N)$ to be the logarithmically-weighted number of representations of $N$ as $N=p+\eta$ where $p$ is a prime, $\eta$ is a square-free number coprime to $k$ and $\eta\neq 1$; namely
\begin{equation} \label{eq: def-R_k_N}
    \overline{R}_k(N) := 
    \sum_{\substack{p \leq n \\ (n-p, k) = 1 \\ p \ne n-1}} \mu^2(n-p) \log p.
\end{equation}
Now, assume GRH and let $N\geq X_3\geq 4\cdot 10^{18}$ be even. Then for any $C\in(0,1/2)$ and positive $B<\sqrt{N}$,
    \begin{align}\label{eveneq}
        \frac{\overline{R_k}(N)}{N}>2c&\prod_{q\mid k/2}\left(1-\frac{q-1}{q^2-q-1}\right)-\frac{c_{\theta}(X_3)\log^2(N)}{\sqrt{N}}\sum_{\substack{d\mid k/2}}\sum_{e\mid d}\sum_{\substack{a\leq B\sqrt{e/d}\\(a,d)=e}}\mu^2(a)\notag\\
        &-E_C(N)\left(\frac{1+2C}{1-2C}\right)\sum_{d\mid k/2}\sum_{e\mid d}\frac{1}{\varphi\left(d/e\right)}\sum_{\substack{a>B\sqrt{e/d}\\(a,d)=e}}\frac{\mu^2(a)}{\varphi(a^2)}\notag\\
        &-\log(N)\left(\sum_{d\mid k/2}\sum_{e\mid d}\left(N^{-\frac{1}{2}}\left(\frac{1}{e}-\frac{1}{d}\right)+\frac{1}{\sqrt{de}}N^{-C}+N^{-2C}\right)\right)\notag\\
        &-\frac{\log(k)}{N}-\frac{\log(N)}{N}.
    \end{align}
    Here, $c=0.37395\ldots$ is Artin's constant, $\mu$ is the M\"obius function, $c_{\theta}(X_3)$ is as in Lemma \ref{cthetalem}, and $E_C(N)$ is defined by
    \begin{equation*}
        E_C(N)=
        \begin{cases}
            0,&\text{if $N^C\leq B$},\\
            1,&\text{if $N^C>B$}.
        \end{cases}
    \end{equation*}
\end{lemma}
\begin{proof}
    If $N^C> B$, the proof is essentially identical to that of \cite[Lemma 5.1]{hathi2024sum} with two main differences. First, the parameter $B$ replaces the choice of $10^5$ used in \cite{hathi2024sum}. Secondly, we have replaced ``$c_{\theta}(da^2/e)/\log n$" with $\frac{c_{\theta}(X_3) \log^2(N)}{\sqrt{N}}$ as a result of the stronger bounds we have under GRH. Note that there is also a slight notation clash with \cite[Lemma 5.1]{hathi2024sum}. Namely, $N$ and $c_\theta$ mean something different in \cite{hathi2024sum} and we have accounted for this accordingly. 

    The only further difference in the case $N^C\leq B$ is that we can omit the term
    \begin{equation*}
        \left(\frac{1+2C}{1-2C}\right)\sum_{d\mid k/2}\sum_{e\mid d}\frac{1}{\varphi\left(d/e\right)}\sum_{\substack{a>B\sqrt{e/d}\\(a,d)=e}}\frac{\mu^2(a)}{\varphi(a^2)}
    \end{equation*}
    appearing in \eqref{eveneq}. This is because in \cite{hathi2024sum} this term appears when considering the range $B<a\leq N^C$, which is empty when $N^C\leq B$.
\end{proof}

\begin{theorem}\label{hathithm}
    Keep the notation and conditions of Lemma \ref{hathilem}, and let $k$ be the product of the first $L+1$ primes. We then have, for $B\geq\max\{45,8\sqrt{k/2}\}$,
    \begin{align*}
        \frac{\overline{R_k}(N)}{N}>2c&\prod_{q\mid k/2}\left(1-\frac{q-1}{q^2-q-1}\right)-\frac{(4+\sqrt{3})^L\cdot B\cdot c_{\theta}(X_3)}{3^L\sqrt{N}}\log^2(N)\\
        &-E_C(N)\left(\frac{1+2C}{1-2C}\right)\cdot 2^L\cdot G\left(\left\lfloor\frac{B}{\sqrt{k/2}}\right\rfloor\right)\\
        &-\log(N)\left(\frac{7^L}{3^L\sqrt{N}}+\frac{(4+\sqrt{3})^L}{3^L N^C}+\frac{3^L}{N^{2C}}\right)\\
        &-\frac{\log(k)}{N}-\frac{\log(N)}{N},
    \end{align*}
    where 
    \begin{equation*}
        G(x)=e^\gamma\left(\frac{\log \log x}{x}-\li\left(\frac{1}{x}\right)\right)+\frac{3}{x}.
    \end{equation*}
\end{theorem}
\begin{proof}
    We write $k'=k/2$ and bound each of the sums from Lemma \ref{hathilem}. First,
    \begin{align*}
        \sum_{\substack{d\mid k'}}\sum_{e\mid d}\sum_{\substack{a\leq B\sqrt{e/d}\\(a,d)=e}}\mu^2(a)\leq \sum_{\substack{d\mid k'}}\sum_{e\mid d}\sum_{\substack{a\leq B\sqrt{e/d}\\e\mid a}}1\leq \sum_{\substack{d\mid k'}}\sum_{e\mid d}\frac{B}{\sqrt{de}}.
    \end{align*}
    To bound this expression further, we note that $k'$ (and each $d\mid k'$) is square-free and odd. Thus, for any $x\mid k'$ with $m$ prime divisors, we have\footnote{One actually has $x\geq \kappa m^m$ for some computable constant $\kappa$. However, we opted for the simpler bound $x\geq 3^m$ as it greatly simplifies the ensuing algebra whilst making very little difference to our final results.} $x\geq 3^m$. So, writing $\omega(d)$ for the number of unique prime factors of $d$,
    \begin{align}\label{bimonsimp}
        \sum_{d\mid k'}\sum_{e\mid d}\frac{1}{\sqrt{de}}&\leq\sum_{d\mid k'}\frac{1}{\sqrt{d}}\sum_{m=0}^{\omega(d)}\frac{1}{(\sqrt{3})^m}\binom{\omega(d)}{m}\notag\\
        &=\sum_{d\mid k'}\frac{1}{\sqrt{d}}\cdot\left(1+\frac{1}{\sqrt{3}}\right)^{\omega(d)}\notag\\
        &\leq\sum_{m=0}^{L}\frac{1}{(\sqrt{3})^m}\left(1+\frac{1}{\sqrt{3}}\right)^m\binom{L}{m}\notag\\
        &=\left(\frac{4+\sqrt{3}}{3}\right)^L,
    \end{align}
    Next,
    \begin{align*}
        \sum_{d\mid k/2}\sum_{e\mid d}\frac{1}{\varphi\left(d/e\right)}\sum_{\substack{a>B\sqrt{e/d}\\(a,d)=e}}\frac{\mu^2(a)}{\varphi(a^2)}&\leq\sum_{a>B/\sqrt{k'}}\frac{\mu^2(a)}{\varphi(a^2)}\sum_{e\mid (a,k')}\sum_{\substack{(a,d)=e\\d\mid k'}}1\\
        &\leq 2^L\sum_{a>B/\sqrt{k'}}\frac{\mu^2(a)}{\varphi(a^2)},
    \end{align*}
    and, by \cite[Theorem 15]{rosser1962approximate},
    \begin{align*}
        \sum_{a>B/\sqrt{k'}}\frac{\mu^2(a)}{\varphi(a^2)}&\leq\sum_{a>B/\sqrt{k'}}\left(\frac{e^{\gamma}\log\log a^2}{a^2}+\frac{2.5}{a^2\log\log a^2}\right)\\
        &\leq\sum_{a>B/\sqrt{k'}}\left(\frac{e^{\gamma}\log\log a^2}{a^2}+\frac{1.76}{a^2}\right)\quad\text{(since $B/\sqrt{k'}\geq 8$)}\\
        &\leq\int_{\lfloor B/\sqrt{k'}\rfloor}^\infty\left(\frac{e^{\gamma}\log\log x}{x^2}+\frac{3}{x^2}\right)\mathrm{d}x\\
        &=G\left(\left\lfloor\frac{B}{\sqrt{k/2}}\right\rfloor\right).
    \end{align*}
    Finally, we want to bound
    \begin{equation}\label{fourthtermeq}
        \sum_{d\mid k/2}\sum_{e\mid d}\left(N^{-\frac{1}{2}}\left(\frac{1}{e}-\frac{1}{d}\right)+\frac{1}{\sqrt{de}}N^{-C}+N^{-2C}\right).
    \end{equation}
    Each term in this double sum is bounded analagously to the double sum in \eqref{bimonsimp}. Namely,
    \begin{align*}
        \sum_{d\mid k'}\sum_{e\mid d}\left(\frac{1}{e}-\frac{1}{d}\right)\leq\sum_{d\mid k'}\sum_{e\mid d}\frac{1}{e}=\left(\frac{7}{3}\right)^L,
    \end{align*}
    \begin{equation*}
        \sum_{d\mid k'}\sum_{e\mid d}\frac{1}{\sqrt{de}}\leq\left(\frac{4+\sqrt{3}}{3}\right)^L
    \end{equation*}
    and
    \begin{equation*}
        \sum_{d\mid k'}\sum_{e\mid d}1= 3^L.
    \end{equation*}
    As a result, \eqref{fourthtermeq} is bounded above by
    \begin{equation*}
        \frac{7^L}{3^L\sqrt{N}}+\frac{(4+\sqrt{3})^L}{3^L N^C}+\frac{3^L}{N^{2C}}
    \end{equation*}
    as desired.
\end{proof}

\begin{corollary}\label{5060cor}
    Assume GRH and let $(\log X_3,L)=(149,20)$ or $(104,12)$. Then every even integer $N\geq X_3$ can be written as the sum of a prime and a square-free number coprime to the product of the first $L+1$ primes.
\end{corollary}
\begin{proof}
    We begin with the case $(\log X_3,L)=(149,20)$. In Theorem \ref{hathithm}, we set $k$ to be the product of the first $L+1=21$ primes, $\log N\geq \log X_3=\exp(149)$, $C = 0.18$, and $B = 10^{22.2}$ to get $\overline{R_k}(N)/N>0$. Note that this computation needs to be done in two parts. Firstly $\exp(149)\leq \log N\leq \frac{\log B}{C}\approx 284$ in which $E_C(N)=0$, then also for $\log N>\frac{\log B}{C}$ in which $E_C(N)=1$. The case $(\log X_3,L)=(104,12)$ is done in the same way with $C=0.15$ and $B=10^{15}$. All computations were done using Python 3.11.3.
\end{proof}

We finish this section with a generalised version of Lemma \ref{djshehzadlem} for which Corollary \ref{5060cor} can be directly applied to.

\begin{proposition}\label{djshehzadprop}
    Let $p_i$ denote the $i^{\text{th}}$ prime and $X_2\geq 4\cdot 10^{18}$. Suppose every even integer $N\geq X_3$ can be written as the sum of a prime and a square-free number coprime to the product of the first $L+1$ primes. Then every even integer $X_3\leq N<X_2$ can be written as the sum of a prime and a square-free number $\eta$ with at most $K$ prime factors, where $K\geq 1$ is the largest integer such that
    \begin{equation*}
        \theta(p_{K+L+1})-\theta(p_{L+1})<\log(X_2).
    \end{equation*}
\end{proposition}
\begin{proof}
    Direct generalisation of the proof of Lemma \ref{djshehzadlem}. 
\end{proof}

\subsection{A conditional lower bound for $S(A,P(z))$}\label{sapsubsect}
We now prove an analogue of Theorem \ref{theo:S>1} assuming GRH. For this we will first need a variant of the Bombieri-Vinogradov theorem (cf. \cite[Lemmas 31 and 33]{BJV22}). 

\begin{lemma}\label{lem: bound E_pi}
    Assume GRH and suppose $B > 0$, $N \geq X_2 \geq 4\cdot 10^{18}$ is even, and $H:= \frac{\sqrt{N}}{\log^{B+1} N}\geq 45$. Then
    \begin{equation*}
        \sum_{\substack{d\le H\\(d,N)=1}}\mu^2(d)\left|\pi(N;d,N)-\frac{\pi(N)}{\varphi(d)}\right|\leq \frac{p_G(X_2)N}{\log^B N},
    \end{equation*}
    where
    \begin{align*}\label{peq}
        p_G(X_2)=0.65\left(c_{\pi}(X_2)+\frac{1}{16\pi}\right) \leq 0.117
    \end{align*}
    with $c_{\pi}(X_2)$ as defined in Lemma \ref{epilem}.
\end{lemma}

\begin{proof}
    First note that we may assume $d\geq 3$ since for $d=1$, we have $|\pi(N;d,N)-\pi(N)/\varphi(d)|=0$, and $d\neq 2$ since $N$ is even. Now, by the triangle inequality
\begin{align*}
    \left|\pi(N;d,N)-\frac{\pi(N)}{\varphi(d)}\right| &\leq \left| \pi(N;d,N) - \frac{\li(N)}{\varphi(d)} \right| + \frac{1}{\varphi(d)} |\li(N) - \pi(N)|.
\end{align*}
We bound the first term using Lemma \ref{epilem}, and by \cite[Corollary 1]{Schoenfeld_76} the second term is bounded above by
\begin{equation*}
    \frac{1}{8 \pi \varphi(d)}\sqrt{N}\log N \leq \frac{1}{16 \pi} \sqrt{N} \log N.
\end{equation*}

Therefore, 
\begin{align*}
    \left|\pi(N;d,N)-\frac{\pi(N)}{\varphi(d)}\right| \leq \left(c_{\pi}(X_2)+\frac{1}{16\pi}\right) \sqrt{N} \log N
\end{align*}
so that, by \cite[Lemma 23]{BJV22}
\begin{align*}
    \sum_{\substack{d\le H\\(d,N)=1}}\mu^2(d)\left|\pi(N;d,N)-\frac{\pi(N)}{\varphi(d)}\right| &\leq 0.65 H\left(c_{\pi}(X_2)+\frac{1}{16\pi}\right) \sqrt{N} \log N\\ 
    &\leq\frac{p_G(X_2)N}{\log^B N},
\end{align*}
as required.
\end{proof}

\begin{lemma}[{cf. \cite[Lemma 43]{BJV22}}]\label{rgrhlem}
Keeping the notation and conditions from Lemma \ref{lem: bound E_pi}, we have
\begin{equation}\label{unsharpeq2}
    \sum_{\substack{d<H\\d\mid P(z)}}|r(d)|<\frac{c_{4,G}(X_2)N}{\log^B N},
\end{equation}
where
\begin{align*}
    c_{4,G}(X_2)&=p_G(X_2)+\frac{0.9}{\sqrt{X_2} \log \log X_2} \leq 0.117.
\end{align*}
\end{lemma}

\begin{proof}
    First we note that, by definition
    \begin{equation*}
        |A|=\pi(N)-\omega(N)\quad\text{and}\quad |A_d|=\pi(N;d,N)-\omega(N;d,N),
    \end{equation*}
    where $\omega(N)$ is the number of distinct prime factors of $N$ and $\omega(N;d,N)$ is the number of distinct prime factors of $N$ that are congruent to $N$ modulo $d$. Hence
    \begin{align*}
        |r(d)|&=\left||A_d|-\frac{|A|}{\varphi(d)}\right|\\
        &=\left|\pi(N;d,N)-\frac{\pi(N)}{\varphi(d)}+\frac{\omega(N)}{\varphi(d)}-\omega(N;d,N)\right|\\
        &\leq \left|\pi(N;d,N)-\frac{\pi(N)}{\varphi(d)}\right|+\omega(N).
    \end{align*}
    Therefore, using Lemma \ref{lem: bound E_pi} and the bounds (\cite[Theorem 11]{robin1983estimation} and \cite[Lemma 23]{BJV22})
    \begin{align*}
        \omega(N)&\leq\frac{1.3841\log N}{\log\log N},\quad n\geq 3\\
        \sum_{d\leq x}\mu^2(d)&\leq 0.65x,\quad x\geq 45,
    \end{align*}
    we have,
    \begin{align*}
        \sum_{\substack{d<H\\d\mid P(z)}}|r(d)|&\leq \frac{p_G(X_2)N}{\log^B N}+0.65H\frac{1.3841\log N}{\log\log N}\\
        &<\frac{N}{\log^B N}\left(p_G(X_2)+\frac{0.9}{\sqrt{X}\log\log X_2}\right)
    \end{align*}
    as required.
\end{proof}

We now give a lower bound for $S(A,P(z))$ assuming GRH.

\begin{theorem}\label{grhsbound}
Assume GRH. Let $X_2>0$, $M\geq 5$, $B\geq 2$, $\alpha > 0$, $u_0\geq 2$, $N \geq X_2$ be even, $z=N^{1/M}$, and $z_0 = X_2^{1/M}$ such that
\begin{equation*}
   \frac{\sqrt{X_2}}{\log^{B+1} X_2}\geq 45,\quad 1-\frac{\xi(z_0,M)}{\log^2N}\geq 0,\quad X_2\geq 4\cdot 10^{18},\quad\epsilon\leq \frac{1}{74}
\end{equation*}
where $\epsilon=\epsilon(u_0,z_0)$ is as in Lemma \ref{1pm1lem}, $\xi(z_0,M)$ is as in Lemma \ref{vjlem}, and
\begin{equation*}
\frac{N^{\alpha}}{\log^{B+1}N}\ge \exp\left(u_0\left(1+\frac{9\cdot 10^{-7}}{\log u_0} \right)\right),~
N^{\frac{1}{2}-\alpha}\ge z^2.
\end{equation*} 
Then,
\begin{align*}
    S(A,P(z))&>M\frac{|A| U_N}{\log^2N}\left(1-\frac{\xi(z_0,M)}{\log^2 N}\right)\\
    &\cdot\Bigg(f\left(M\left(\frac{1}{2}-\alpha\right)\right)-C_1(\epsilon)\epsilon e^2h\left(M\left(\frac{1}{2}-\alpha\right)\right)\\
    &-\frac{1}{M}\left(1-\frac{\xi(z_0,M)}{\log^2 N}\right)^{-1}\left(2e^{-\gamma}\prod_{p>2}\left(1-\frac{1}{(p-1)^2}\right)\right)^{-1}\frac{c_{4,G}(X_2)}{\log^{B-2}N}\Bigg),
\end{align*}
where $C_1(\varepsilon)$ is from \cite[Table 1]{BJV22}, $c_{4,G}(X_2)$ is defined in Lemma \ref{rgrhlem}, and all other notation is as in Section \ref{sectnot}.
\end{theorem}
\begin{proof}
    We argue similarly to the case $k_1<K_{\delta}(x_1)$ in the proof of \cite[Theorem 44]{BJV22}. So, set the parameters as in \eqref{def: D-s-Q}.
    Since $D\ge z^2$, we have by \cite[Theorem 6]{BJV22}
    \begin{equation}\label{smallsieveeq}
        S(A,P(z))>\frac{M|A|U_N}{\log N}\left(1-\frac{\xi(z_0,M)}{\log^2N}\right)(f(s)-C_1(\epsilon)\epsilon e^2h(s))-\sum_{\substack{d\mid P(z)\\ d<QD}}|r(d)|.
    \end{equation}
    We now remark that the condition 
    \begin{equation*}
        \frac{N^{\alpha}}{\log^{B+1}N}\ge \exp\left(u_0\left(1+\frac{9\cdot 10^{-7}}{\log u_0}\right)\right)
    \end{equation*}
    implies that
    \begin{equation}\label{qeq}
        Q\le \frac{N^{\alpha}}{\log^{B+1} N}
    \end{equation}
    by \cite[Lemma 25]{BJV22}. As a result, $QD\le H:=\frac{\sqrt{N}}{\log^{B+1}N}$ so that we may apply Lemma \ref{rgrhlem} to \eqref{smallsieveeq}. This gives the desired result upon noting that $|A|>N/\log N$ (\cite[Lemma 42]{BJV22}).
\end{proof}
Equipped with this conditional lower bound on $S(A,P(z))$, we now finally prove Theorem \ref{conthm}.
\begin{proof}[Proof of Theorem \ref{conthm}]
    We set $X_2=\exp(\exp(5.077))$, $\alpha=0.285$, $B=2.01$, $u_0=30$ and $M=18$. With these choices of parameters $\epsilon=1.312\cdot 10^{-2}$ (Lemma \ref{1pm1lem}), $C_1(\epsilon)=559$, $\xi(z_0,M)\leq 1244$ (Lemma \ref{vjlem}) with $z_0=\exp(8.9)$) and
    \begin{equation*}
        f\left(M\left(\frac{1}{2}-\alpha\right)\right)=f(3.87)=0.97044\ldots\quad\text{(see \eqref{fs24})}.
    \end{equation*}
    Applying Theorem \ref{grhsbound} we then obtain
    \begin{equation*}
        S(A,P(z))>\frac{1.48U_NN}{\log^2 N}>0,
    \end{equation*}
    where we used that $|A|>N/\log N$ (Lemma \ref{Alem}).
    This means that, assuming GRH, every even $N\geq\exp(\exp(5.077))$ can be written as the sum of a prime and a number with at most $M-1=17$ prime factors.
    
    For $\exp(149)\leq N<\exp(\exp(5.077))$, we apply Corollary \ref{5060cor} and Proposition \ref{djshehzadprop} with $L+1=21$ to prove that $K=31$ works in this range. For ${\exp(104)\leq N<\exp(149)}$ we use the $L+1=13$ result in Corollary \ref{5060cor} and Proposition \ref{djshehzadprop} to obtain that $K = 31$ also works in this range. Finally, for $2<N<\exp(104)$ we use Lemma \ref{djshehzadlem} and obtain the even better value $K = 25$ in this range. Combining each case gives that $K=31$ works for all $N\geq 4$ as required.
\end{proof}

\section{Possible improvements}\label{improvsec}
With more work, it should be possible to improve our main results (Theorems \ref{unconthm} and \ref{conthm}). There are many avenues to do this, so in what follows we detail what we believe are some of the most impactful approaches. If the reader is interested in pursuing any of these avenues, the authors are very open to correspondence on the matter.

Before we begin, a general point is that we expect many of the explicit results that go into our proof to improve naturally in line with increased computational power. So in this regard, we remark that extending the computations of Platt \cite{platt2016numerical} regarding zeros of Dirichlet $L$-functions, would be a sure-fire way to improve the ingredients used for the unconditional result (Theorem \ref{unconthm}).

\subsection{Bounds on primes in arithmetic progressions}\label{PNTAPimprov}
The main bottleneck to improving the unconditional result is our existing bounds on the error term in the prime number theorem for arithmetic progressions. In our approach, we used the recent bounds obtained by Bordignon in \cite{Bordignon_21}. Certainly, one could get a small improvement in our results by extending Table 6 in \cite{Bordignon_21} to give more optimal parameters. However, on inspection, it appears that there are several other aspects of Bordignon's work that can be improved. 

Firstly, in \cite{Bordignon_21} the error term in the explicit formula \cite[(1)]{Bordignon_21} is obtained using a method due to Goldston \cite{Goldston_83}. However, an asymptotically better error term can be obtained from the work of Wolke \cite{Wolke_1983} and Ramar\'e \cite{Ramare_16_Perron}. An explicit form of such an error term was recently obtained by Cully-Hugill and the first author \cite{cully2023error, johnston2024error}.

Moreover, the zero-free regions for Dirichlet $L$-functions could be improved. Namely, there is recent work of Kadiri \cite{kadiri2018explicit} which could be built upon to give better bounds on Siegel/exceptional zeros compared to \cite[Theorem 1.1]{Bordignon_21}. This would also lead to a better (i.e.\ lower) value of $R_1$ that could be used in this work. The methods used in the recent work of Morrill and Trudgian \cite{morrill2020elementary} could also be useful in this regard.

\subsection{Explicit bounds on Siegel zeros}
In addition to the bounds one can obtain on Siegel zeros described in Section \ref{PNTAPimprov}, we also seek to improve bounds of the form
\begin{equation}\label{generalsiegelbound}
    \beta_q\leq 1-\frac{\lambda}{\sqrt{q}\log^2 q}
\end{equation}
where $\lambda$ is a positive constant, and $\beta_q$ is a (potential) Siegel zero mod $q$. This bound is that which appears in \eqref{siegelbound} and is an important component in the proof of Theorem \ref{unconthm}. For $q>4\cdot 10^5$, Bordignon \cite{bordignon2019explicit,bordignon2020explicit} shows that one can take $\lambda=100$ and this is what we use. Here, we note that for $q\leq 4\cdot 10^5$ there are no Siegel zeros by a computation due to Platt \cite{platt2016numerical}. However, the relevant computation in \cite{platt2016numerical} was only a side result of the main computation, meaning a more targeted approach could pay dividends\footnote{In fact, through private correspondence, Platt and Trudgian claim to have shown that there are no Siegel zeros for $q\leq 10^9$. However, this work has yet to be published.}. 

It also appears that the factor of $\log^2q$ can be removed from \eqref{generalsiegelbound} by using an approach due to Goldfeld and Schinzel \cite{goldfeld1975siegel}. This has already be done for odd characters in \cite{ralaivaosaona2020explicit} but a version that also works for even characters would be required in our setting.

\subsection{Bounds on sums and products of primes}
Another key component which goes into our results are bounds on 
\begin{equation*}
    \sum_{p<x} 1/p\qquad\text{and}\qquad\prod_{u\leq p<z}\left(1-\frac{1}{p-1}\right)^{-1}
\end{equation*}
(see Lemmas \ref{mertenlem} and \ref{1pm1lem}). Some of these bounds could be greatly improved by computation. For instance, if one were to extend the computation used for \cite[Lemma 16]{BJV22} to all $2\leq x\leq 10^{13}$, then the constant ($2.964\cdot 10^{-6}$) appearing in the bound
\begin{equation*}
     \sum_{p<x}\frac{1}{p}\geq\log\log x+M-\frac{2.964\cdot 10^{-6}}{\log x}
\end{equation*}
would be reduced to $1.483\cdot 10^{-6}$.

\subsection{Further exploration of sieve methods}
Throughout recent history there have been numerous sieve-theoretic approaches to the problem of expressing large even numbers as the sum of a prime and an almost prime. The overarching sieve used in this paper is the linear sieve, and we use the explicit version from \cite[\S 2]{BJV22}. One could further explore the existing literature on linear sieves (e.g.\ \cite{halberstam2011sieve,friedlander2010opera,lichtman2023primes}) and likely find an approach that is superior to the one here.

In this direction we also remark, as discussed in Section \ref{sectuncon}, that our approach works for at best $K=4$ prime factors. It would be interesting to explore simpler methods, such as those using Brun's sieve (e.g.\ \cite[\S 2.4]{halberstam2011sieve}), which fail asymptotically for such low values of $K$, but might give better explicit results than those in Theorems \ref{unconthm} and \ref{conthm}.

\printbibliography

\end{document}